\documentclass[12pt]{article}

\usepackage{amsfonts}

\usepackage{amscd}
\usepackage{amsthm}
\usepackage{indentfirst}

\usepackage[all]{xy}

\usepackage{amssymb, amsmath, amsthm, amsgen, amstext, amsbsy, amsopn}
\usepackage{epsfig}
\usepackage{epic,eepic}

\usepackage{verbatim}

\newtheorem{thm}{Theorem}[section]
\newtheorem{lemma}[thm]{Lemma}
\newtheorem{prop}[thm]{Proposition}
\newtheorem{cor}[thm]{Corollary}
\theoremstyle{definition}\newtheorem{defin}[thm]{Definition}
\newtheorem{rem}[thm]{Remark}
\newtheorem{exm}[thm]{Example}

\newtheorem{question}[thm]{Question}

\newcommand{\z}[1][X] {{\mathcal O_{#1}}}

\newcommand{\id}{\mathbbm 1}

\newcommand{\R}{\mathbb R}
\newcommand{\N}{\mathbb N}

\newcommand{\Z}{\mathbb Z}
\newcommand{\T}{\mathbb T}
\newcommand{\eps}{\varepsilon}
\newcommand{\sub}{\subset}
\newcommand{\eset}{\emptyset}

\newcommand{\wtil}{\widetilde}
\newcommand{\tcG}{\widetilde{\mathcal{G}}}

\newcommand{\cB}{{\mathcal{B}}}

\newcommand{\cG}{{\mathcal{G}}}

\newcommand{\cU}{{\mathcal{U}}}

\newcommand{\wcU}{{\widehat{\mathcal{U}}}}

\def\id{{1\hskip-2.5pt{\rm l}}}

\newcommand{\Stab}{{\hbox{\rm Stab}}}

\newcommand{\supp}{{\it supp\,}}

\DeclareMathOperator{\sgrad}{sgrad}

\begin{document}

\author{\textsc Maia Fraser\thanks{Department of Mathematics and Statistics, University of Ottawa},
Leonid Polterovich
\thanks{School of Mathematical Sciences, Tel Aviv University} and Daniel Rosen\footnotemark[2]}

\title{On Sandon-type metrics for contactomorphism groups}

\date{\today}

\maketitle

\abstract{For certain contact manifolds admitting a 1-periodic Reeb flow we construct a conjugation-invariant norm on the universal cover of the contactomorphism group. With respect to this norm the group admits a quasi-isometric monomorphism of the real line. The construction involves the partial order on contactomorphisms and symplectic intersections. This norm descends to a conjugation-invariant norm on the contactomorphism group. As a counterpoint, we
discuss conditions under which conjugation-invariant norms for contactomorphisms are necessarily bounded.

\renewcommand{\thefootnote}{\fnsymbol{footnote}} 
\renewcommand{\thefootnote}{\arabic{footnote}} 

\section{Introduction}\label{sec:intro}

A \emph{conjugation-invariant norm} on a group $G$ is a function $\nu \colon G \to [0, \infty)$ satisfying the following properties:
\begin{enumerate}
\item $\nu(\id) = 0$ and $\nu(g) > 0$ for all $g \neq \id$.
\item $\nu(gh) \leq \nu(g) + \nu(h)$ for all $g, h \in G$.
\item $\nu(g^{-1}) = \nu(g)$ for all $g \in G$.
\item $\nu(h^{-1}gh) = \nu (g)$ for all $g, h \in G$.
\end{enumerate}
A function $\nu$ satisfying only 1-3 is a \emph{norm} on $G$ , while if the non-degeneracy condition $\nu(g) > 0$ for $g \neq \id$ is dropped $\nu$ is said to be a \emph{pseudo-norm}.

Given any bi-invariant metric $d$ on $G$, distance to the identity defines
a conjugation-invariant norm, $\nu(f) : = d(f, \id),$ and vice versa, any conjugation-invariant norm $\nu$
defines a bi-invariant metric $d(f, g) := \nu(fg^ {-1})$.

Following the terminology of \cite{BIP}
we say a norm on $G$ is  \emph{bounded} when there exists $C < \infty$
such that $\nu(g) \leq C$ for all $g \in G$. A norm is called {\it stably unbounded} if for some
$g \in G$, $\nu(g^n) \geq c|n|$ for all $n \in \Z$ with some $c>0$. A norm $\nu$ is \emph{discrete} if there exists a constant $c> 0$ such that $c \leq \nu(g)$ for any $g \neq \id$, and a norm is \emph{trivial} if it is both discrete and bounded (i.e., equivalent to the trivial norm). In many cases we will consider, a general argument of \cite{BIP} implies that all conjugation-invariant norms are discrete, and hence boundedness is equivalent to triviality.

In this paper, we focus on conjugation-invariant norms on contactomorphism groups and in particular on their (un)boundedness. 
Such norms were discovered by S. Sandon in \cite{sandon-bi-invt}
and further studied in recent papers \cite{zap} by F. Zapolsky and \cite{colin-sandon} by V. Colin and S. Sandon. Their geometric properties
turn out to be sensitive to the contact topology of $(V,\xi)$. The above norms are:
\begin{itemize}
\item unbounded for $T^*\R^n \times S^1$, but not stably unbounded \cite{sandon-bi-invt};
\item stably unbounded for $T^*X \times S^1$ with compact $X$ \cite{zap, colin-sandon} and for $\R P^{2n+1}$; \cite{colin-sandon};
\item bounded for $S^{2n+1}$ \cite{colin-sandon},
\end{itemize}
where the manifolds in the list are equipped with the standard contact structures.
All these norms are studied by using Legendrian spectral invariants. Sandon's norm and its extension
by Zapolsky are actually defined through these invariants, while the Colin-Sandon norms have geometric and/or dynamical definitions. It is also possible to define not conjugation-invariant (pseudo-) norms purely in terms of Hamiltonians; one of these, in the spirit of Hofer's norm, has been studied by Shelukhin \cite{shelukhin} (cf. \cite{Ryb12, Ryb13}). 

Conjugation-invariant norms are closely related to
 quasi-morphisms. Indeed if $\mu$ is a homogeneous quasi-morphism on a group $G$ --
 a real-valued function on $G$ such that $\phi(f^n) = n\phi(f)$ for any $f \in G$, $n \in \Z$ and for which there is $D>0$ such that $|\phi(fg) - \phi(f)- \phi(g)| \leq D$ for all $f,g \in G$ -- then it is easily checked that $\phi$ is conjugation-invariant and $\mu(f):=|\phi(f)| + D$ for $f \neq \id$,
 $\mu(\id) := 0$ defines a conjugation-invariant norm on $G$. In particular the quasi-morphisms which S. Borman and F. Zapolsky
 \cite{BZ} construct on prequantizations give rise to stably unbounded conjugation-invariant
 norms.

Finally, we note that, as with diffeomorphisms,
a fragmentation
property holds for contactomorphisms (see Banyaga \cite{banyaga})
and any open cover therefore induces a corresponding
fragmentation norm
(c.f. \cite{BIP} for diffeomorphisms,
\cite{colin-sandon} for contactomorphisms). With some care in the choice of cover
this norm will be conjugation-invariant and, as we show in Section~\ref{subsect-upper}, will dominate all known
conjugation-invariant norms, a situation analogous to that for
diffeomorphism groups \cite{BIP}.

\medskip

The aim of the present paper is twofold. In Section~\ref{sec:constr}
for certain contact manifolds admitting a 1-periodic Reeb flow we give yet another construction of a stably unbounded conjugation-invariant norm on the (universal cover of) contactomorphism groups. The construction involves the partial order on contactomorphism groups introduced in \cite{EP}, and the stable unboundedness
of the norm is deduced from basic results on symplectic intersections. The examples include various prequantization spaces such as $T^*X \times S^1$ with closed $X$, prequantizations of symplectically aspherical manifolds containing a closed Bohr-Sommerfeld Lagrangian submanifold and the standard projective spaces $\R P^{2n+1}$.

These results are contrasted with the following statement proved in Section~\ref{sec:obstr}:
if the contact fragmentation norm is bounded, which in particular holds for
$V = S^{2n+1}$, then
every known conjugation invariant norm on the identity component of the contactomorphism group is trivial. The proof follows closely \cite{BIP}.

\section{Constructions}\label{sec:constr}

\subsection{Preliminaries}

Let $(V,\xi)$ be a contact manifold, not necessarily closed, with co-oriented contact structure $\xi$.
Let us fix some notation for the groups we will be dealing with.
We write $\cG(V, \xi)$ for the identity component of
the group of compactly supported contactomorphisms of $(V, \xi)$. This is shortened
to $\cG(V)$ or $\cG$ when clear from the context.
We denote by $\tcG$ the universal cover of $\cG$.

Contact isotopies supported in a given open set $X \subset V$ give rise to subgroups of $\cG(V)$ and $\tcG(V)$ generated by these isotopies. We denote them respectively by $\cG(X) \subset \cG(V)$
and $\tcG(X,V) \subset \tcG(V)$. Let us mention that in general $\tcG(X,V)$ does not coincide
with the universal cover $\tcG(X)$ of $\cG(X)$. However there exists a natural
epimorphism
\begin{equation}\label{eq-chi}
\tcG(X) \to \tcG(X,V)\;.
\end{equation}

Throughout this section we assume that $\lambda$ is a contact form which obeys the co-orientation and whose Reeb vector field
generates a circle action $e_t$, $t \in S^1$. As an auxiliary group - not the primary object of our study - we
define
$\cG_e(V, \lambda)$ to be the group of contactomorphisms
of the form $e_t \cdot \phi$ where
$\phi \in \cG(V)$. We emphasize that by {\it contactomorphism},
we always mean a diffeomorphism preserving the contact structure, and in this last case, the co-orientation.
We do not, however, require contactomorphisms to preserve any specific contact form.
We denote by $\tcG_e$ the universal cover of $\cG_e$.

Observe that when $V$ is a closed manifold, $\cG(V)=\cG_e(V)$.  When $V$ is an open manifold,
every $f \in \cG_e(V)$ coincides with some $e_t$ outside a sufficiently large compact subset, so we
have a fibration $\cG(V) \to \cG_e(V) \to S^1$. The exact homotopy sequence yields in this case
that
$$0=\pi_2(S_1) \to \pi_1(\cG) \to \pi_1(\cG_e)\;,$$
and hence $\pi_1(\cG) \to \pi_1(\cG_e)$ is a monomorphism. This implies that $\tcG$ can be considered
as a subgroup of $\tcG_e$.

Let $SV= (V \times \R_+, d(s\lambda))$ be the symplectization of $V$. For a contactomorphism $f$ of $V$
we write $\bar{f}$ for the corresponding $\R_+$-equivariant symplectomorphism of $SV$.
A time-dependent function $F_t:SV \to \R$ which is $\R_+$-equivariant, i.e. such that $F_t(sx) =  sF_t(x)$ for all $s \in \R_+, x \in SV$, is
called a {\it contact Hamiltonian}. The Hamiltonian flow
it defines is also $\R_+$-equivariant and so produces a contact isotopy of $(V, \lambda)$. Moreover, every
contact isotopy $f_t$ is given uniquely by such an $F_t$. In particular, the Reeb flow $e_t$ is given by the
(time-independent) contact Hamiltonian $s$.

We consider
the stabilization $SV \times T^*S^1$ of $SV$ with the symplectic form $d[s\lambda + rdt])$, and for
a subset $X \subset SV$
denote $\Stab (X) = X \times \z[S^1] \subset SV \times T^*S^1$,
where in general in this paper we write $\z[M]$ for the zero section of $T^*M$. Under the
diffeomorphism
$\sigma: S(V \times T^* S^1) \rightarrow SV \times T^* S^1$,
$(s, u, r, t) \mapsto (s, u, sr, t)$ for $s \in \R_+, u \in V, r \in \R, t \in S^1$,
the Liouville form $s\lambda + rdt$
on the exact symplectic manifold $SV \times T^\ast S^1$ pulls back to the form
$s(\lambda + rdt)$ on $SV \times T^*S^1$, therefore $(V \times T^* S^1, \lambda + rdt)$ is a contact
manifold whose symplectization is symplectomorphic $SV \times T^*S^1$ and we have an induced
$\R_+$-action on
$SV \times T^*S^1$ given by
$c \cdot (z, r, t) = (c \cdot z, cr, t)$ where $z \in SV$, $r \in \R$ and $t \in S^1$.

\medskip
\noindent\begin{defin}\label{defin-stab-int}{\rm
We say that a compact set $B \subset SV$ has {\it stable intersection property} if $\Stab(B)$ cannot be
displaced from $SV \times \z[S^1]$ by an $\R_+$-equivariant Hamiltonian diffeomorphism of $SV \times T^*S^1$.}
\end{defin}

\medskip

Any $f \in \tcG_e$ is a homotopy class
of path connecting the identity to a fixed contactomorphism. We write $\{ f_t\}_{t \in [a,b]}$ for a specified path in the class $f$ and simply $\{ f_t\}$ when $[a, b] = [0, 1]$. We remind that the contact Hamiltonian
depends on this choice of path within the class $f$ and is not uniquely defined by $f$. We write $H(f_t)$ to denote the Hamiltonian for the path $\{ f_t \}$.
Recall that the cocycle formula yields
\begin{equation}\label{eq-cocycle}
H(f_tg_t)=H(f_t)+ H(g_t)\circ (\bar{f}_t)^{-1}\;.
\end{equation}

\medskip
The following binary relation introduced and studied in \cite{EP} plays a crucial role in our story:
we write $f \succeq \id$, $f \in \tcG_e$
if $f$ ``can be given by a non-negative contact Hamiltonian'', i.e. there is some path $\{f_t\}$ in the class of $f$
having a non-negative contact Hamiltonian.
Observe that in this case
${d \over dt} f_t(x) \in T_{f_tx}V$ belongs to the non-negative half-space bounded by
the contact hyperplane $\lambda_{f_tx}$.
We remark that having a non-negative Hamiltonian is a coordinate-free condition and so $f \succeq \id$ is invariant under conjugation of $f$ in $\tcG_e$.
We write $f \succeq g$ if $fg^{-1} \succeq \id$
(here $fg^{-1}$ corresponds to the class of paths $\{f_t g_t^{-1}\}$). By conjugating with $g^{-1}$, we have
the equivalent definition: $f \succeq g$ if $g^{-1}f \succeq \id$. Observe that the relation $\succeq$
is reflexive.

We denote by $e$ the lift of $\{e_t\}_{t \in [0,1]}$ to $\tcG_e$, and by $e^c$, $c \in \R$
the element of $\tcG_e$ represented by the path $\{e_t\}_{t \in [0,c]}$.

\begin{lemma}\label{lem-succeq}
\begin{enumerate}
\item \label{item-FG} For $f,g \in \tcG_e$,  $f \succeq g$ if and only if $f$ and $g$ can be given by Hamiltonians
$F$ and $G$ such that $F \geq G$. Moreover one can prescribe either $F$ or $G$
in advance, hence $\succeq$ is transitive;
\item $f \succeq g$ and $a \succeq b$ implies that $fa \succeq gb$, i.e. $\succeq$ is bi-invariant;
\item  Any element $\phi \in \tcG_e$ generated by a positive contact Hamiltonian bounded away from zero on the hypersurface $\{s=1\}$ is \emph{dominant}, i.e. $\forall f \in \tcG_e, \exists p \in \N$ s.t.
$\phi^p \succeq f$. In particular any $e^c, \, c > 0$ is dominant.
\end{enumerate}
\end{lemma}

\begin{proof}
To prove the {\it if}~ direction of the first property, fix paths $\{f_t\}$ and $\{g_t\}$ for $f$ and $g$ such that $H(f_t) \geq H(g_t)$. By \eqref{eq-cocycle}
\begin{align*}
H(g_t^{-1}f_t)(z, t)
&= H(g_t^{-1})(z, t) + H(f_t)(\bar g_tz, t) \\
&= - H(g_t)(\bar g_tz, t) + H(f_t)(\bar g_tz, t).
\end{align*}
So $\{g_t^{-1}f_t\}$ is a path in the class $g^{-1}f$ having non-negative Hamiltonian.
Now, to prove the {\it only if} direction, let $\{h_t\}$ be a path in the class $g^{-1}f$
having non-negative Hamiltonian $H(h_t)$. Let $\{g_t\}$ be an arbitrary path for $g$ and set $f_t = g_th_t$
(or if we wish to prescribe $f_t$ then set $g_t$ accordingly).
Then $h_t = g_t^{-1}f_t$ and so the earlier computation shows that $H(f_t) \geq H(g_t)$.

The second property is also proved using the cocycle formula \eqref{eq-cocycle} (comparing
Hamiltonians for $f (ab^{-1})$ and $g$). The third property is straightforward
(note that $cs$ is a contact Hamiltonian for $e^c$).
\end{proof}

\begin{rem}\label{rem:periodic}
The Hamiltonians $F$ and $G$ in Lemma~\ref{lem-succeq} item \ref{item-FG}
can moreover be taken to be 1-periodic.
This is because a non-negative (resp. positive) isotopy can be homotoped
within the class of non-negative (resp. positive) isotopies to one with
1-periodic Hamiltonian. For positive isotopies this is
given by Lemma \cite[3.1A]{EP}, see also the argument within the proof
of Theorem \cite[1.19]{EKP06}.
The same arguments go through for non-negative isotopies.
\end{rem}

The next result is a direct analogue of Theorem \cite[2.3A]{EP}.
Reflexivity, transitivity and bi-invariance of $\succeq$ are given by the previous Lemma.
As in  \cite[2.3A]{EP}, the only issue is anti-symmetry.

\medskip
\noindent \begin{thm} \label{thm-1} Suppose that $SV$ contains a compact set $B$
with stable intersection property.
Then $\succeq$ is a bi-invariant partial order on $\tcG_e$.
\end{thm}

\begin{rem}\label{rem:anti-sym}
	Anti-symmetry of $\succeq$ on the universal cover of a group of contactomorphisms
	fails if and only if that group contains a
	non-constant contractible loop with non-negative Hamiltonian.
	For closed $V$, $V$
	is said to be \emph{orderable}\footnote{In fact, although the study of $\succeq$ on $\tcG(V)$ was initiated in \cite{EP}, the terminology
	orderable was introduced later in \cite{EKP06}.} \cite{EP} when $\succeq$
	is anti-symmetric on $\tcG(V)$. For open $V$
	with 1-periodic Reeb flow we consider a stronger notion,
	referring to $V$ as \emph{orderable}
	when $\succeq$ is anti-symmetric on $\tcG_e(V)$.
\end{rem}

The proof of Proposition \cite[2.3A]{EP} derives orderability
from the stable intersection property using two results: Proposition \cite[2.1B]{EP} and Proposition \cite[2.3B]{EP}.
The first applies
to $\cG(V)$ for closed $V$ and says that existence of a non-negative, non-constant contractible loop implies
existence of
a positive one. The second shows the
stable intersection property prevents existence of a positive
contractible loop.

The first proposition \cite[2.1B]{EP} requires compactness of $V$, but can be modified to deal with
$\cG_e(V)$ as shown in Proposition~\ref{prop:posContr}.
The second proposition \cite[2.3B]{EP} generalizes immediately
to our setting and is stated and proved as Lemma \ref{lem-loop} below.
These two results together imply Theorem \ref{thm-1} (given compact $B \subset SV$
we let $C$ be its projection to $V$).

\begin{prop} (c.f. Proposition {\cite[2.1B]{EP}})  \label{prop:posContr}
Assume there exists a non-negative non-constant loop
$\{f_t\}_{t \in S^1}$ of contactomorphisms.
Then for any compact set $C \subset V$ there exists a loop $\{g_t\}_{t \in S^1}$ of contactomorphisms
whose contact Hamiltonian $G_t$ is positive on $SC$ for all $t \in S^1$. Moreover,
when $\{f_t\}_{t \in S^1}$ is contractible so is $\{g_t\}_{t \in S^1}$. Compact support is also retained.
\end{prop}

\begin{proof}
At first we proceed as in the first two steps of the proof of Proposition {\cite[2.1B]{EP}}, and then we
resort to a modification.

As in step 1 of {\cite[2.1B]{EP}}, without loss of generality
we may assume $F(z, 0) \neq 0$ for some $z \in SV$.
Let  $U \subset V$ be such that
$F(z, 0) > 0$
for all $z \in SU$.

As in step 2 of {\cite[2.1B]{EP}}, we take a sequence $\varphi_1, \ldots, \varphi_d$ of
elements of $\cG_e$ such
that $C \subset \bigcup\limits_{k = 0}^d \psi_k (U),$
where
$\psi_0 = \id$, $\psi_k = \varphi_1 \cdots \varphi_k$ for $k = 1, \ldots, d$.
This can be done since $C$ is compact.
Define
\begin{equation}
g_t = f_t\varphi_1 f_t \ldots \varphi_d f_t (\varphi_1 \cdots \varphi_d)^{-1}. \label{eq:spread}
\end{equation}
This forms a loop $\{g_t\}_{t \in S^1}$
generated by the Hamiltonian
\begin{align*}
G(z, t) = F(z, t) + &F(\bar\varphi_1^{-1} \bar f_t^{-1} z, t) + \\
\ldots + &F(\bar \varphi_d^{-1} \bar f_t^{-1} \cdots \bar\varphi_1^{-1} \bar f_t^{-1} z, t).
\end{align*}
For $z \in SU$ the first summand is positive when $t = 0$. On the other hand,
for $z \in SC \setminus SU$ there exists $k$ such that $\bar \psi_k^{-1}z \in SU$, in
which case the $k$'th summand is positive when $t = 0$.
Since all summands are
non-negative we conclude
$G(z, 0) > 0$ for all $z \in SC$. Note that $\{g_t\}_{t \in S^1}$ is contractible if $\{f_t\}_{t \in S^1}$ is,
and compact support is also retained, so without loss of generality we now
assume $F(z, t) > 0$ for all $z \in SC$, $t = 0$ and hence
for all $z \in SC$ and $t \in \Delta \subset S^1$ a closed interval containing $0$.

Let $H$ be an autonomous contact Hamiltonian $H \colon SV \to \R$
which is positive on the set $\cup_{t\in S^1} \bar f_t (SC)$. Moreover assume
$H$ vanishes outside $SK$ for some compact $K \subset V$ so the associated contact isotopy $h_t$
is compactly supported.
By construction, $H(\bar f_t^{-1}z)>0$ for all $t \in S^1$, $z \in SC$.
We claim there exists a smooth function $u \colon [0,1] \to \R$ with the following properties:
\begin{itemize}
\item $u(0)=u(1)=0$.
\item For all $t \notin \Delta$, $u'(t)>0$.
\item For all $t \in \Delta$,
\begin{equation}\label{eq-u'}
u'(t)>-\min_{z\in SC}\frac{F(z,t)}{H(\bar f_t^{-1}z)}.
\end{equation}
\end{itemize}
Indeed, the minimum \eqref{eq-u'} exists and is positive, since for each $t \in \Delta$, ${F(z,t)}/{H(\bar f_t^{-1}z)}$ is a well defined, positive, and $\R_+$-invariant function on $SC$; therefore, we can allow $u'$ to be negative on part of $\Delta$
and so make it positive outside of $\Delta$.

Consider now the loop $g_t=f_t h_{u(t)}$. Clearly, it is homotopic to $\{f_t\}$ via the endpoint-preserving homotopy $\{f_th_{su(t)}\}$. It is also compactly supported when $f_t$ is (because $\{h_t\}$ has compact support).

Moreover, its Hamiltonian is $G(z,t)=F(z,t)+u'(t)H(\bar f_t^{-1}z)$ which, by condition \eqref{eq-u'},
is positive for all $z \in SC$ and $t \in \Delta$.
We claim that $G(z,t)>0$ for all $z \in SC$ and $t \in S^1$. Indeed,
if $t \notin \Delta$, then $u'(t)>0$ and, as observed above, $H(\bar f_t^{-1}z)>0$
for $z \in SC$, thus implying $G(z,t)>0$ (since $F \geq 0$).
\end{proof}

\subsection{The norm $\nu$}
\noindent {\bf Throughout this section we assume that $V$ is {orderable}; that is, the relation
$\succeq$ is a partial order on $\tcG_e(V)$ (see Remark~\ref{rem:anti-sym}).}

\medskip
\noindent

\medskip\noindent
Observe that $e$ lies in the center of $\tcG_e$.
For an element $f \in \tcG_e$ consider the following invariants:
$$\nu_+(f) := \min\{n\in \Z\;:\: e^n \succeq f\}$$
and
$$\nu_-(f) := \max\{n\in \Z\;:\: e^n \preceq f\}\;.$$
Note that $\nu_- \leq \nu_+$ by transitivity of $\succeq$.

\medskip\noindent
It is readily checked that $\nu_+$ and $\nu_-$ are conjugation-invariant (since $e$ is in the center) and
$\nu_-(f) = -\nu_+(f^{-1})$, using the bi-invariance of $\succeq$.
We also observe that $\nu_+$ and $\nu_-$ are respectively sub- and super-additive:
\begin{align*}
\nu_+(fg) &\leq \nu_+(f) + \nu_+(g) \\
\nu_-(fg) &\geq \nu_-(f) + \nu_-(g).
\end{align*}
We write
$$\nu(f) := \max (|\nu_+ (f)|, |\nu_- (f)|)\;.$$

\begin{rem} \label{rem-amplitude} Let $k \in \Z_{\geq 0}$. We remark that
	$\nu(f) \leq k$ if and only if $e^{-k} \preceq f \preceq e^k$. This
	in turn is equivalent to the property:
	$f$ can be generated by a contact Hamilonian $F_+$
	such that $F_+ \leq ks$ and also by a contact Hamiltonian $F_-$ such that
	$F_- \geq -ks$ (see Lemma~\ref{lem-succeq}).
	By Remark~\ref{rem:periodic}, the Hamitonians $F_\pm$ may moreover
	be assumed to be 1-periodic.
\end{rem}

We have the following.

\medskip
\noindent \begin{thm} \label{thm-2} $\nu$ is a conjugation-invariant norm on $\tcG_e$.
\end{thm}

\begin{proof}
Clearly $\nu \geq 0$.
Observe moreover that $\nu(f)=0$ if and only if $\id \succeq f \succeq \id$ and hence $f= \id$.
Since both $\nu_+$ and $\nu_-$ are conjugation-invariant, $\nu$ is as well.
It remains to prove the triangle inequality. Observe that, by Remark \ref{rem-amplitude},
$$
\nu(f) = \min \{k \in \Z_{\geq 0} \,:\, e^{-k} \preceq f \preceq e^k \}.
$$
Now, let $f,g \in \tcG_e$, and put $m = \nu(f)$, $n=\nu(g)$. Then
$$
e^{-m} \preceq f \preceq e^m; \quad e^{-n} \preceq g \preceq e^n.
$$
By bi-invariance of $\succeq$, this implies $e^{-(m+n)} \preceq fg \preceq e^{m+n}$, which gives $\nu(fg) \leq m+n=\nu(g)+\nu(g)$.
\end{proof}

\begin{rem} Note
the above argument in fact proves that
$\nu$ is a conjugation-invariant semi-norm which
is non-degenerate if and only if $\succeq$ is anti-symmetric.
By trivial reasons, non-degeneracy of the semi-norm $\nu$ is therefore
equivalent to orderability of $\tcG_e(V)$ (see Remark~\ref{rem:anti-sym}).
For closed $V$, the oscillation semi-norm $\nu_{osc}$ of \cite{colin-sandon}
has an analogous property (c.f. Proposition \cite[3.2]{colin-sandon}):
by definition it is non-degenerate
if and only if there is no positive
contractible loop in $\cG(V)$, and
this in turn is equivalent
to orderability by Proposition \cite[2.1B]{EP}.
\end{rem}

\begin{rem} Note that $a \succeq b \succeq \id$ implies $\nu(a) \geq \nu(b)$.
In other words, $(\tcG(V), \nu, \succeq)$ is a partially ordered metric space in the sense
of \cite{EP}, Section 1.7.
\end{rem}

\medskip
\noindent \begin{thm}\label{thm-3} Suppose that a compact set $B \subset SV$ has stable intersection property
and in addition $B$ is invariant under the flow  $\bar{e}_t$. Assume
$f \in \tcG_e$ can be generated by a contact Hamiltonian $F_t$ satisfying
$F_t > cs$ on $B$ for some $c >0$ and all $t \in S^1$.
Then $\nu_+(f) \geq [c]$, the integer part of $c$,
and so $\nu(f) \geq [c]$.
\end{thm}

\begin{cor}\label{cor-stb-unbd-tG} Suppose that a compact set $B \subset SV$ has stable intersection property
and in addition $B$ is invariant under the flow  $\bar{e}_t$.
Then $\nu$ is stably unbounded on $\tcG(V) \subset \tcG_e(V)$.
Moreover, $\tcG(V)$ admits a quasi-isometric monomorphism of the real line.
\end{cor}

\begin{proof} Suppose $c \in \N$ and $\epsilon > 0$. Let $C = \pi(B) \subset V$
be the (compact) projection of $B$. Define $F$ to be an autonomous contact Hamiltonian
equal to $(c + \epsilon) s$
on $SC$ and supported in a larger $SK$, with $K \subset V$ compact. Then $F$ generates an element $f \in \cG(V)$
which satisfies the hypotheses of Theorem~\ref{thm-3}.
In particular if $n\in \Z$ then
$nF$ is a Hamiltonian for $f^n$. For
$n > 0$ it strictly exceeds $cn$ on $B$ so the Theorem gives
$\nu_+(f^n) \geq cn$, while for $n <0$
we obtain $\nu_-(f^n) \geq c|n|$. We conclude $\nu(f^n) \geq c|n|$ for
all $n \in \Z$. Moreover, taking $\epsilon <1$,
we can arrange that $|F| \leq (c+1)s$ on $SV$, which by Remark~\ref{rem-amplitude} gives
$\nu(f^n) \leq (c+1)|n|$. Since $F$ is autonomous
its Hamiltonian flow $t \mapsto f_t$ thus
defines a quasi-isometric monomorphism $\R \to \tcG(V)$.
\end{proof}

\begin{rem}\label{rem-stb-unbd-tGe}
Observe that, by definition, $\nu(e^n) = n$ for any $n \in \N$, and hence $\nu$ is always stably unbounded on $\tcG_e$. The point of Corollary \ref{cor-stb-unbd-tG} is that existence of a compact invariant set with the stable intersection property implies stable unboundedness also on $\tcG$. On the other hand, when $V$ is compact, $\tcG_e(V) = \tcG(V)$ and hence $\nu$ is always stably unbounded on $\tcG(V)$.
\end{rem}

\medskip
\noindent Before proving the theorem, we recall the following construction (see \cite{EP}).
Let $\varphi= \{\phi_t\}_{t \in S^1}$, $\phi_0=\phi_1=\id$ be a loop of contactomorphisms in $\cG_e(V)$
generated by a contact Hamiltonian $\Phi_t$ on $SV$. Define {\it the suspension map}
$$\Sigma_\varphi : SV \times T^* S^1 \to SV \times T^* S^1$$
of $\varphi$ by
$$ (z,r,t) \mapsto (\bar{\phi}_tz, r-\Phi_t(\bar{\phi}_tz),t)\;.$$
The map $\Sigma_\varphi$ is an $\R_+$-equivariant symplectomorphism of $SV \times T^*S^1$.
Given two loops $\varphi$ and $\theta$, the co-cycle formula implies $\Sigma_{\varphi\circ\theta} = \Sigma_\varphi \circ \Sigma_\theta$. Furthermore, if
$\varphi$ is contractible and $\varphi^{(s)}$ is the homotopy of $\varphi=\varphi^{(0)}$ to the constant
loop $\varphi^{(1)}=\id$, the family of suspension maps $\Sigma_{\varphi^{(s)}}$ is a Hamiltonian isotopy
of $SV \times T^*S^1$.

\medskip
\noindent
\begin{lemma} (c.f. Proposition {\cite[2.3B]{EP}}) \label{lem-loop}
Let $B \subset SV$ be a compact set with stable intersection property.
Then for every contractible loop $\varphi=\{\phi_t\}_{t \in S^1}$ its contact Hamiltonian $\Phi$ vanishes for some
$t_0 \in S^1$ and $y \in B$: $\Phi_{t_0}(y)=0$.
\end{lemma}

\begin{proof} The stable intersection property implies that the sets $\Sigma_\varphi(SV \times \z[S^1])$ and $B \times \z[S^1]$ intersect. Thus there exist $z \in SV$ and $t_0 \in S^1$ such that $\bar{\phi}_{t_0}z \in B$ and
$\Phi_{t_0}(\bar{\phi}_{t_0}z)=0$. Setting $y= \bar{\phi}_{t_0}z$, we get the lemma.
\end{proof}

\medskip
\noindent {\bf Proof of Theorem \ref{thm-3}:}
Without loss of generality we may assume $c$ to be an integer.
Suppose it is not true that $\nu_+(f) \geq c$.
Then $f \preceq e^c$ (recall $e^c$ denotes the class of the path $\{e_{ct}\}$).
This means that
\begin{equation}\label{eq-varphi}
H(\bar f_t) \leq H(\bar e_{ct}\bar\phi_t)
\end{equation}
for some contractible loop $\varphi=\{\phi_t\}$ on $\cG_e$.
By the cocycle formula, this yields
$F_t \leq cs + \Phi_t \circ \bar e_{-ct}$. Applying Lemma \ref{lem-loop}, we see that
$\Phi_{t_0}(y)=0$ for some $t_0 \in S^1$ and $y \in B$. Since $x:=\bar e_{ct}y\in B$, we get that
$F_t(x) \leq cs(x)$, contradicting the assumption $F_t|_B > cs$.
\qed

\medskip
\noindent If we consider only compactly supported contactomorphisms, i.e. restrict the above norm to $\tcG$, it descends to $\cG$ as follows. Given $f' \in \cG$, define
\begin{equation}\label{eq-exact-1}
\nu_*(f') := \inf \nu(f)\;,
\end{equation}
where the infimum is taken over all lifts $f$ of $f'$ to $\tcG \subset \tcG_e$. Observe that $\nu_*$ is non-degenerate: indeed, since $\nu(f)$ is integer, the infimum is necessarily attained on some lift $f$, but
$\nu(f)=0$ yields $f=\id$ and hence $f'=\id$.

Denote by $\Pi$ the image of the fundamental group $\pi_1(\cG,\id)$ in $\pi_1(\cG_e,\id)$
under the natural inclusion morphism. Each loop in $\cG_e$ representing an element of $\Pi$ can be written
as a product of a contractible loop in $\cG_e$ and a loop in $\cG$ (note that the order of factors
is not important since $\cG$ is a normal subgroup of $\cG_e$).

Given a compact subset $B \subset SV$, we say that it has {\it strong stable intersection property} if for every loop $\varphi$ representing an element of $\Pi$
its suspension $\Sigma_\varphi$  satisfies

\begin{equation}\label{eq-exact}
\Sigma_\varphi (SV \times \z[S^1]) \cap (B \times \z[S^1]) \neq \emptyset\;.
\end{equation}

\medskip
\noindent
\begin{thm}\label{thm-contract}
Suppose that a compact set $B$ has strong stable intersection property
and in addition $B$ is invariant under the flow  $\bar{e}_t$. Assume
$f \in \tcG$ is generated by a contact Hamiltonian $F_t$ satisfying
$F_t > cs$ on $B$, for some $c >0$ and all $t \in S^1$. Let $f' \in \cG$ be the
time-one endpoint of $f$. Then
$\nu_* (f') \geq c$.
\end{thm}

\medskip
\noindent
The proof repeats verbatim that of Theorem~\ref{thm-3} with the following modifications:
Lemma \ref{lem-loop} extends to
any loop $\varphi=\{\phi_t\}$ representing an element of $\Pi$ under the assumption of
strong stable intersection property, and given $f$ and $f'$ as in our hypotheses, inequality
\eqref{eq-varphi} holds for a loop $\varphi=\{\phi_t\}$ representing an element of $\Pi$.

\subsection{Examples}

In this section we discuss some settings where our norm is well-defined.

\begin{rem}
In Examples~\ref{exam-1}, \ref{exam-AF}, \ref{exam-usher} below
we consider $V$ of the form
$T^*X \times (S^1)^k$, $k \in \N$. In all these settings,
given a set $Y \subset T^*X$ we
write $\widehat Y:= Y \times (S^1)^k$ to denote its lift to $V$.
\end{rem}

\medskip
\begin{exm}\label{exam-1}
Assume that $V= T^*X \times S^1$, where $X$ is a closed manifold, $\lambda = d\tau - pdq$, $e_t(p,q,\tau)=(p,q,\tau+t)$.
Consider the Lagrangian submanifold
$$B := \z[X] \times \{s=1\} \subset SV,$$
where $\{s=1\} \subset S^1 \times \R_+$. $B$
is stably non-displaceable by standard Floer theory. Thus $\nu$ is a metric on
$\tcG_e$ in this case. Furthermore, Theorem \ref{thm-contract} is applicable to this situation. Indeed
we claim (see proof below) for any loop
$\varphi = \{\phi_t\}$ in $\cG_e$ representing an element of $\Pi$, $\Sigma_\varphi(\Stab(B))$
has the same Liouville class as $\Stab(B)$
and hence these Lagrangian submanifolds intersect by a theorem
of Gromov \cite[$2.3.B_4''$]{G85},
implying $B$ has strong stable intersection property.
We conclude that the norm $\nu_*$ defined by \eqref{eq-exact-1}
is unbounded on the subgroup $\cG(\wcU,V)$ consisting of all contactomorphisms
generated by contact isotopies with support in $\wcU$, where $\cU \subset T^*X$
is any tube containing the zero section. In fact, the group $\cG(V)$ is stably unbounded
with respect to $\nu_*$. Let us mention also that the norm $\nu_*$ is greater than or equal to
the norm defined by Zapolsky in \cite{zap} (this readily follows from \cite{zap}).

\begin{proof}
We need to show
$[\Lambda|_{\Stab(B)}] = [\Lambda|_{\Sigma_\varphi(\Stab(B))}]$ where $\Lambda := s\lambda + rdt$ is the Liouville
form on $SV \times T^*S^1$. Without loss of generality assume $\varphi = \{\phi_t\}$
is a loop on $\cG$ (since for contractible $\varphi$ the suspension $\Sigma_\varphi$
is Hamiltonian).
By the K\"unneth formula $H_1(\Stab(B))$ is generated by
loops of the form $\{(\gamma(\rho), 0, 0) \}_{\rho \in S^1}$ and $\{(z_0, 0, t)\}_{t \in S^1}$
where $\{\gamma(\rho)\}_{\rho \in S^1}$ is a loop
in $B$ and $z_0 \in B$. Thus $H_1(\Sigma_\varphi (\Stab(B)))$ is generated by loops
$\{(\gamma(\rho),  - \Phi_0(\gamma(\rho)), 0) \}_{\rho \in S^1}$ and
$\{(\bar\phi_t z_0,  - \Phi_t(\bar \phi_t z_0) , t)\}_{t \in S^1}$.
$\Lambda$ coincides on loops $\{(\gamma(\rho), 0, 0) \}_{\rho \in S^1}$
and $\{(\gamma(\rho),  - \Phi_0(\gamma(\rho)), 0) \}_{\rho \in S^1}$
(since $rdt$ vanishes). Before comparing the other loops
note that we may deform $\z[X] \times S^1$ by
a contact isotopy so that at least one point $w$ of the image lies outside the compact
support of the isotopy $\phi_t$. This deformation lifts and extends trivially to a Hamiltonian isotopy $h_t$ of
$SV \times T^*S^1$ so without loss of generality, appealing also to the Hamiltonian isotopy
$\Sigma_\varphi h_t \Sigma_\varphi^{-1}$ for $\Sigma_\varphi(\Stab(B))$, we may replace
$z_0$ by $(s, w)$ for some $s \in \R_+$ when evaluating $\Lambda$ on loops $\{(z_0, 0, t)\}_{t \in S^1}$ vs.
$\{(\bar\phi_t z_0,  - \Phi_t(\bar \phi_t z_0) , t)\}_{t \in S^1}$.
These loops now lie outside the support of $\Sigma_\varphi$ and so coincide.
\end{proof}
\end{exm}

\medskip
\begin{exm}\label{exam-AF} As above, let  $V= T^*X \times S^1$.
Under extra hypotheses, one can make an even stronger statement than just unboundedness of our norm
on $\tcG(\wcU,V)$ for a tube $\cU \subset T^*X$ about the zero section (cf. Example~\ref{exam-1}). We now
describe a setting in which $\tcG(\wcU,V)$ with norm $\nu$ admits a quasi-isometric monomorphism of $\R^N$  for any $N$.
Let $L \subset T^\ast X$ be a closed Lagrangian submanifold such that
\begin{itemize}
\item [(a)] $HF(L, L) \neq 0$ (Floer homology with coefficients in a field, say $\Z_2$)
\item [(b)] $(a\cdot L) \cap L = \emptyset, \; \forall a > 0, a \neq 1$, where $a \cdot (p, q) = (ap, q)$.
\end{itemize}
We claim that for any bounded domain $\mathcal U \subset T^\ast X$ containing the zero
section and any $N \in \N$,~
$\tcG(\wcU,V)$ admits a quasi-isometric monomorphism of $\R^N$. Some examples of
$X$ and $L$ as above:
\begin{enumerate}
\item $X$ is a closed manifold admitting a closed $1$-form $\alpha$ without zeroes,
and $L$ is the graph of $\alpha$.
\item $X = S^2$ and $L$ is the Lagrangian torus studied by Albers and Frauenfelder in \cite{AF08}
with $HF(L, L; \Z_2) \neq 0.$
\end{enumerate}
\begin{proof}
Fix $N \in \N$ and fix a bounded tube $\mathcal U \subset T^*X$ around the zero section.
Choose distinct real numbers $a_1, \ldots, a_N, a_j \neq 1$
such that $L_j := a_jL \subset \mathcal U$. Thus $L_j, \; j=1, \ldots, N$
are pairwise disjoint.
We now
identify $SV$ with a domain $\mathcal W := \Theta(SV) \subset T^*X \times T^*S^1$ via the
$\R_+$-equivariant symplectic embedding
$$\Theta: SV \to T^*X \times T^*S^1,
(p,q, s, \tau) \mapsto (-s \cdot p,q,s,\tau)$$
where $(p, q) \in T^*X, s \in \R_+, \tau \in S^1$.
In $\mathcal W$, put $\widehat L_j := L_j \times S^1 \subset V$.
Let $W_j \subset T^*X$ be tubular neighborhoods of $L_j$ respectively such that
$\overline{W_j} \cap \overline{W_i} = \emptyset$
when $i \neq j$, where $\overline{W_j}$
denotes the closure of $W_j$. Then the closures of their lifts $\widehat {W}_j$ to $V$
are pairwise disjoint.

Take contact Hamiltonians $H_j$, with $\supp H_j \subset
S(\widehat {{W}_j})$, $H_j = 1$ on $\widehat L_j$ and $0 \leq H_j \leq s$.
Let $h^t_j$ be the corresponding Hamiltonian flow. Consider the map
$\Psi: \R^N \to \tcG(\wcU,V)$
given by $(t_1, \ldots, t_N) \mapsto h^{t_1}_1 \dots h^{t_N}_N$.
This is by construction a homomorphism, which is injective since $W_j$'s are pairwise disjoint.
On the one hand, $h^{t_1}_1 \dots h^{t_N}_N$
is generated by $H = \sum\limits_{j=1}^Nt_jH_j$ so
$$\nu(h^{t_1}_1 \dots h^{t_N}_N) \leq \max\limits_j (|t_j| + 1).$$
On the other hand, $H|_{\widehat L_j} = t_j$ so by Theorem \ref{thm-3}
$$\nu(h^{t_1}_1 \dots h^{t_N}_N) \geq \max\limits_j (|t_j| - 1).$$
Thus
$$ ||t||_\infty - 1 \leq \nu(h^{t_1}_1 \dots h^{t_N}_N) \leq ||t||_\infty + 1\;,$$
where $||t||_\infty:= \max_j|t_j|$. We conclude that
$$\Psi:(\R^ N, ||\cdot||_\infty) \to (\tcG(\cU,V), \nu)\;$$
is a quasi-isometric monomorphism.
\end{proof}
\end{exm}

\medskip
\noindent
\begin{exm}\label{exam-usher}
Let $X$ be a closed manifold equipped with a Riemannian metric $\rho$ {\it without contractible geodesics}.
For $c >0$ put
$$\Xi_c:= \{(p,q) \in T^*X \;:\; |p|_\rho=c\;\}.$$
Let $\T^k = (S^1)^k$ denote the $k$-torus.
We claim that for every $c>0$ and $k \geq 1$, the subset
$$\Xi_c \times \T^k \subset T^*X \times T^*\T^k$$
(identifying $\T^k$ with the zero section $\z[\T^k]$) is non-displaceable.
As an immediate consequence, arguing as in Example \ref{exam-AF},
we get that for any bounded domain $\mathcal U \subset T^\ast X$ containing the zero
section and any $N \in \N$,~
$\tcG(\wcU,V)$ admits a quasi-isometric monomorphism of $\R^N$. We thank Michael Usher \cite{usher} for
his suggestion to consider hypersurfaces $\Xi_c$ in a similar Hofer-geometric context.

\begin{proof} Observe that $\Xi'= \Xi_c \times \T^k$ is a coistropic submanifold of $T^*X \times T^*\T^k$.
Moreover, it is stable in the sense of Theorem 1.5 of \cite{Ginzburg}. Assume on the contrary that
$\Xi'$ is displaceable. By Ginzburg's Theorem 1.5, there exists a disc of positive symplectic area
with boundary lying on one of the fibers of $\Xi'$. Every fiber of $\Xi'$ is of the form $L:=\gamma \times \T^k$, where $\gamma$ is a trajectory of the geodesic flow on $\Xi$. Since all closed geodesics of $\rho$ are non-contractible,
the inclusion $L \to T^*X \times T^*\T^k$ induces a monomorphism
of fundamental groups. Thus every disc with boundary on $L$ has vanishing symplectic area, a contradiction.
\end{proof}
\end{exm}

\medskip
\noindent
\begin{exm}\label{exm-prequantization} {\rm Let $(M,\omega)$ be a closed symplectic manifold
with $[\omega] \in H^2(M,\Z)$. Let $\pi: V \to M$ be a prequantization of $M$. This means
that $\pi$ is a principal $S^1$-bundle equipped with an $S^1$-invariant contact form $\lambda$
such that $d\lambda =\pi^*\omega$. The Reeb flow $e_t$ of $\lambda$ is just the natural $S^1$-action
on $(M,\omega)$. We shall focus on the group $\tcG(V)=\tcG_e(V)$ (these groups coincide since
$M$ is closed).

Assume that $M$ contains a closed Lagrangian submanifold $L$ with the following properties:
\begin{itemize}
\item[{(i)}] The connection on $V$
defined by $\lambda$ has trivial holonomy when restricted to $L$ (the Bohr-Sommerfeld condition);
\item[{(ii)}] The relative homotopy group $\pi_2(M,L)$ vanishes.
\end{itemize}
It is proven in \cite[Theorem 1.3.D]{EP} that under these assumptions $\succeq$ is a partial order on $\tcG(V) = \tcG_e(V)$ (these groups coincide, as $V$ is closed) and we deduce that the norm $\nu$ is well-defined on $\tcG(V)$. Moreover, since $V$ is compact, by Remark \ref{rem-stb-unbd-tGe}, $\nu$ is stably unbounded.}
\end{exm}

\begin{exm}\label{exm-projective} Let $V = \R P^{2n+1}$ with standard contact structure. It is proven in \cite[Theorem 1.3.E]{EP} that $\succeq$ is a partial order on $\tcG(\R P^{2n+1}) = \tcG_e(\R P^{2n+1})$. Therefore, the norm $\nu$ is well defined on $\tcG(\R P^{2n+1})$. Moreover, since $\R P^{2n+1}$ is compact, $\nu$ is unbounded, by Remark \ref{rem-stb-unbd-tGe}. The situation changes drastically when we pass to the double cover
	$S^{2n+1}$ of $\R P^{2n+1}$: as we shall see in the next section, any conjugation-invariant norm on $\tcG(S^{2n+1})$
	is bounded, provided $n \geq 1$.

While $\nu$ descends to a conjugation-invariant norm $\nu_*$ on $\cG(\R P^{2n+1})$,
it is not clear if $\nu_*$ is unbounded.
Could it be that every conjugation-invariant norm on
$\cG(M^3)$, $M^3$ a contact three-manifold, is bounded? The analog of this statement holds for diffeomorphisms by a result of Burago-Ivanov-Polterovich (Theorem \cite[1.11(iii)]{BIP}) and work in progress by Patrick Massot seeks to use open book decompositions to develop a contact version of
that argument.
\end{exm}

\subsection{Norm $\nu$ and $k$-translated fixed points}

In this section we remark on the relationship of our norm $\nu$ -- on $\tcG$ -- with the notion of translated point.
We say that $f \in \tcG$ has a {\it $k$-translated fixed point}
$x \in V$, $k \in \N$, if there exists a contact isotopy $\{f_t\}$, $t \in [0,1]$ in the class $f$ such that $\bar f_tx =\bar e_{kt}x$. In particular, $f_1 x = x$, since the Reeb flow is 1-periodic.

This is a definition for
$f$ in the universal cover $\tcG$. In terms of Sandon and Colin's terminology
\cite{sandon-thesis, colin-sandon} in $\cG$ it amounts to saying
that there is a path $\{f_t\}$ in the class $f$
such that for all $t \in [0, 1]$, $x$ is a {\it translated point} of $f_t$
with $e_{kt}x$ being the point at which the $f_t$ orbit
of $x$ re-joins the Reeb chord through $x$.
Since $k \in \N$, $x$ is actually a {\it discriminant point} of $f_1$
(i.e. a translated point which is fixed).
We remark that translated points in Sandon's terminology also correspond to {\it leafwise intersection points}
which have been studied recently by Albers-Merrry \cite{albers-merry} as well as Sandon \cite{sandon-leaf}.

We claim the notion of $k$-translated fixed points
for $f \in \tcG$ is
related to the norm $\nu$ in the following sense: contactomorphisms in the ball of $\nu$-radius
$k$ which do not remain in that ball when perturbed necessarily have $k$-translated fixed points.
We now make the notion of perturbation precise.

Define, for $\eps > 0$, $\mathring{\cB}(\eps)$ (resp. $\cB(\eps)$) as the set of $f\in \tcG$ which can be generated by Hamiltonians $F_\pm$ such that $F_+(x,s,t)< \eps s$ and $F_-(x,s,t)>-\eps s$
(resp. $F_+(x,s,t)\leq \eps s$ and $F_-(x,s,t)\geq-\eps s$).
As in Remark~\ref{rem:periodic} we may assume such Hamiltonians
are 1-periodic.
Recall that for integer $\eps = k$, $\cB(k)$ is the ball of radius $k$ in the norm $\nu$, i.e.
$$
\cB(k)=\{f\in \tcG \colon \nu(f) \leq k\}.
$$
We say that $f\in \tcG$ is \emph{$k$-robust}, for $k \in \N$, if
$f \mathring{\cB}(\eps) \subset \cB(k)$ for some $\eps >0$.
Note that $f$ is $k$-robust if and only if $gfg^{-1}$ is $k$-robust
(though the value of $\epsilon$, in general, will change).

One readily checks that $f$ is $k$-robust if and only if $f\in \mathring{\cB}(k)$. Indeed, for the `only if' direction, let $\epsilon > 0$ such that $f\mathring{\cB}(\epsilon) \subset \cB(k)$ and take
positive $c < \epsilon$ and elements $g_\pm \in \mathring{\cB}(\epsilon)$ such that $g_\pm$ coincides with $e^{\pm c}$ on the support of $f$. Then $fe^{\pm c} \in \cB(k)$ yields
the existence of Hamiltonians $F_\pm$ as needed to conclude $f \in \mathring{\cB}(k)$.

\medskip
\noindent
\begin{thm}\label{pro-tpoints} Let $k\in \N$. Every $f \in \cB(k)$ without $k$-translated points is $k$-robust.
\end{thm}

\begin{proof} Assume $f \in \cB(k) \subset \tcG$ has no $k$-translated fixed points.

\medskip
\noindent {\sc Step 1.} Let $W := SV \times T^*S^1$, with coordinates $(s, x,r,t)$. Recall that $W$ is equipped with the symplectic form $\Omega:= d(s\lambda) + dr \wedge dt$ and
with the $\R_+$-action $c \cdot (s,x,r,t) = (cs, x,cr,t)$, $c \in \R_+$.

Since $f \preceq e^k$, $f$ is generated by a 1-periodic
Hamiltonian $F(s,x, t)$ with $F(s,x, t) \leq ks$. Moreover $F$ vanishes when $x$ lies outside some compact subset of $V$.

Put $H(s,x, r,t) = r + F(s,x,t)$ and $K(s,x,r,t) = r+ks$.
Since $H \leq K$ the hypersurface $\Xi = \{H =0\} \subset W\;$ lies in the closed domain $U = \{K \geq 0\} \subset W\;$. Moreover, $dH=dK$ at each point of the set $Y= \Xi \cap \partial U$, since on $Y$ the function $H-K$ attains its maximal value. Thus the Hamiltonian vector fields $\sgrad H$ and $\sgrad K$ coincide on $Y$.

Observe that all orbits of the Hamiltonian flow of $K$ on $\partial U$ are (up to time shifts $\tau \mapsto \tau + \tau_0$, where $\tau$ stands for the time variable of the flow) circles of the form $\gamma_{x,s}(\tau)=  (e_{k\tau} x, s, -ks, \tau)$.

\medskip
\noindent {\sc Step 2.} We claim that {\it the set $Y \subset \Xi$
does not contain a compact invariant set of the Hamiltonian flow $h_\tau$ of $H$ on $\Xi$.} Indeed, otherwise this invariant set necessarily contains some $h_\tau$-orbit, which, since $\sgrad H=\sgrad K$ on $Y$, must be a circle of the form $\gamma_{x,s}(\tau)$.
Denote by $p: W \to SV$ the natural projection, and note that
$$ \bar{f}_\tau (x,s) = p\bigl(h_\tau(x,s,-ks,0)\bigr) =  p(\gamma_{x,s}(\tau)) = \bar{e}_{k\tau}(x,s)\;.$$ Therefore
$(x,s)$ is a $k$-translated fixed point of $f$, a contradiction with the assumption of the theorem. The claim follows.

\medskip
\noindent {\sc Step 3.} Observe that the Hamiltonians $H,K$ are equivariant with respect to the $\R_+$-action on $W$. Since $H(x,s,r,t)=r$ for $x$ outside a compact subset of $V$ and
$K(x,s,r,t) = r+ks$ with $k >0$, the set $Y/\R_+$ is compact.
By an $\R_+$-equivariant application of a theorem of Sullivan \cite{Sul,LS2} there exists an $\R_+$-equivariant function $\Phi(x,s,r,t)$ on $W$
with $d\Phi (\sgrad H) <0$ at every point of $Y$ (namely, apply Sullivan's theorem to the induced flow on $\{s = 1\}$ identified with $W / \R_+$ and extend the resulting function equivariantly). Here we use the fact that $Y$ does not contain a compact invariant set of the Hamiltonian flow $h_\tau$ on $\Xi$, see Step 2.
Since on $Y$
$$d\Phi (\sgrad H) =d\Phi(\sgrad K)= \Omega(\sgrad K, \sgrad \Phi) =-dK(\sgrad \Phi ),$$
it follows that $\sgrad \Phi$ is transversal to $\partial U$ at the points of $Y$ and moreover $\sgrad \Phi$ looks inside $U$ at points of $Y$ (cf. the proof of \cite[Theorem 1.5]{PPS-MMJ}).
Denoting by $\phi_t$ the Hamiltonian flow of $\Phi$, we get that for a sufficiently small $\epsilon >0$
\begin{equation}
\label{eq-strict-ineq}
\phi_\epsilon(\Xi) \subset \text{Interior} (U)\;.
\end{equation}

\medskip
\noindent {\sc Step 4.} The hypersurface $\Xi':= \phi_\epsilon(\Xi)$ is transversal to the lines parallel  to the $r$-axis and hence has the form
$\Xi'=\{r+F'(x,s,t)=0\}$ for some $\R_+$ equivariant Hamiltonian $F'$ on $SV$. Put $H'(x,s,r,t)=r+F'(x,s,t)$. We claim that the time one map $f'$ of $F'$ is conjugate to $f$ in $\tcG(V)$. Indeed, $S = \Xi \cap \{t=0\}$ is a Poincar\'e section of
the Hamiltonian flow $h_\tau$ on $\Xi$. Similarly, $S' = \Xi' \cap \{t=0\}$ is a Poincare section of the Hamiltonian flow $h'_\tau$ of $H'$ on $\Xi'$. Denote by $\psi$ and $\psi'$ the corresponding return maps.

Further, $\phi_\epsilon(S)$ is a Poincar\'e section of $f'_t$ with return map
\begin{equation}\label{eq-conj-1}
\psi''=\phi_\epsilon \psi \phi_\epsilon^{-1}\;.
\end{equation}

The orbits of $\bar{f}'_t$ establish an $\R_+$-equivariant symplectomorphism, say $\eta$, between $S'$ and $\phi_\epsilon(S)$.
Thus
\begin{equation}\label{eq-conj-2}
\psi' = \eta^{-1}\psi''\eta\;.
\end{equation}

Finally, let $\pi: S \to SV$ and $\pi': S' \to SV$ be the restrictions of the natural projection. Then
$\bar{f}= \pi\psi\pi^{-1}$ and $\bar{f}' = \pi'\psi'(\pi')^{-1}$. Combining this with \eqref{eq-conj-1}
and \eqref{eq-conj-2} we get that $\bar{f}$ and $\bar{f}'$ are conjugate by $\R_+$-equivariant symplectomorphisms, and hence $f$ and $f'$ are conjugate as well.

\medskip
\noindent {\sc Step 5.} By \eqref{eq-strict-ineq} we have the strict inequality $F' < ks$,
and hence $f'\mathring{\cB}(\delta) \subset \cB(k)$
for some $\delta >0$. Since $f'$ is conjugate to $f$,
the same holds for $f$ (with, perhaps, smaller $\delta$).

Repeating the arguments of Steps 1-5, and decreasing if necessary $\delta>0$ we get that
$f\mathring{\cB}(\delta) \subset \cB(k)$. This yields robustness of $f$.
\end{proof}


\section{Obstructions}\label{sec:obstr}

\subsection{Overview}
In the next sections we discuss some restrictions on conjugation-invariant norms on $\cG(V)$ or
$\tcG(V)$ for certain contact manifolds $V$. Our first result concerns discreteness, our second result boundedness. Recall that a conjugation invariant norm $\mu$ is called discrete if $\mu(g) \geq c$ for some $c>0$ and all $g \neq \id$.

\begin{thm}\label{thm-discrete}
	Let $V$ be any contact manifold.
	\begin{enumerate}
		\item Any conjugation-invariant norm on $\cG(V)$ is discrete.
		\item Any conjugation-invariant norm on $\tcG(V)$ is discrete on $\tcG(V) \setminus \pi_1 (\cG(V))$.
	\end{enumerate}
\end{thm}

\begin{exm}
This example shows the second part of Theorem \ref{thm-discrete} cannot be improved.
It follows from \cite{E92} that $\pi_1(\cG(S^3)) = \Z$. Let $\phi \in \pi_1(\cG(S^3))$ be a generator, and let $r \in (0,1)$ be an irrational number. Define a norm $\mu$ on $\tcG(S^3)$ by setting
$$\mu(\phi^n) = |e^{2\pi i n r}- 1|,$$ and $\mu(g) = 1$ for $g \notin \pi_1(\cG(S^3))$.
One readily checks that $\mu$ defines a norm on $\tcG(S^3)$, which is conjugation-invariant since
$\pi_1(\cG(S^3))$ is a normal subgroup. Moreover, $\mu$ is clearly not discrete.
\end{exm}

Next we address boundedness. To this end, we consider the fragmentation norm. Recall that any compactly supported contact isotopy can be represented as a finite product of contact isotopies each supported in a Darboux ball (see \cite{banyaga}). Here by Darboux ball we mean a contact embedded
image of an open ball centred at the origin in the standard Euclidean space. The contact fragmentation norm $\nu_F(f)$ of $f \in \tcG(V)$ is the minimal number of factors in such a representation of $f$. One can analogously define the contact fragmentation norm on $\cG(V)$. These norms are useful for us as they are maximal in the following sense:

\begin{thm}\label{thm-frag-universal}
Let $V$ be a contact manifold and let $\mu$ be a conjugation-invariant norm on $\cG(V)$ or $\tcG(V)$ which is bounded on a $C^1$-neighborhood of the identity. Then there is a constant $C = C(V,\mu)$ such that $\mu \leq C \cdot \nu_F$.
\end{thm}

This automatically implies:

\begin{cor}\label{cor-frag-bdd}
Let $V$ be a contact manifold and suppose the fragmentation norm on $\cG(V)$ (resp. $\tcG(V)$) is bounded. Then any conjugation- invariant norm on $\cG(V)$ (resp. $\tcG(V)$) which is bounded on a $C^1$-neighborhood of the identity is bounded.
\end{cor}

As an example, consider the sphere $S^{2n+1}$ with its standard contact structure, for $n\geq 1$.

\begin{prop}\label{prop-frag-sphere-bdd}
The fragmentation norm on $\tcG(S^{2n+1})$ is bounded by $2$ when $n\geq 1$.
\end{prop}

\begin{proof}
For $z \in S^{2n+1}$ put $V_z = S^{2n+1} \setminus \{z\}$. Observe that $V_z \subset S^{2n+1}$ is a Darboux ball. Now, let $\{f_t\}$ be a contact isotopy representing $f \in \tcG(S^{2n+1})$. Take a sufficiently small ball  $B \subset V$ such that $X := \cup_t f_t(B) \neq S^{2n+1}$. Fix any point $z \notin X$.  Let $\{g_t\}$ be a contact isotopy supported in $V_z$ with $\left.g_t \right|_{B} = \left.f_t \right|_{B}$ for all $t \in [0,1]$. Set $h_t = g_t^{-1}f_t$. Observe that $h_t \in \cG(V_w)$ for any point $w \in B$. Then $f = gh$ and hence $\nu_F(f) \leq 2$.
\end{proof}

As an immediate consequence of Proposition \ref{prop-frag-sphere-bdd} and Corollary \ref{cor-frag-bdd} we get:
\begin{cor}
Let $n \geq 1$. Any conjugation-invariant norm on $\cG(S^{2n+1})$ or $\tcG(S^{2n+1})$ which is bounded on a $C^1$-neighborhood of the identity is bounded.
\end{cor}

\begin{rem}
All boundedness results in this paper - both Theorem \ref{thm-frag-universal} and Corollary \ref{cor-frag-bdd} above and the analogous resuts in
Section~\ref{subsection-sandtubes} -
involve a $C^1$-boundedness hypothesis and their proofs
use perfectness of groups of contactomorphisms of finite
smoothness \cite{Tsuboi}. If one instead appeals to perfectness of the
group of smooth contactomorphisms \cite{Ryb} then this additional hypothesis
is not needed and one obtains analogous boundedness statements which
hold for any conjugation-invariant norm.
We note, nevertheless,
that all the norms mentioned in this paper do
satisfy the $C^1$-boundedness assumption appearing above:
\begin{itemize}
\item The norm $\nu$ (and consequently also $\nu_*$) defined in Section \ref{sec:constr} satisfies this assumption, since any element of $\tcG(V)$ sufficiently $C^1$-close to $\id$ can be represented by a flow generated by a Hamiltonian satisfying $|H| \leq \epsilon s$.
	\item Zapolsky's norms $\rho_{\mathrm{osc}}$ and $\rho_{\mathrm{sup}}$ \cite{zap} satisfy this assumption since, as mentioned in Example \ref{exam-1}, they are bounded above by the norm $\nu_*$ (as follows from \cite[Proposition 2.9(iii)]{zap}).
	\item The discriminant, zig-zag, and oscillation norms \cite{colin-sandon} satisfy this assumption. Indeed, both the oscillation and discriminant norms are bounded above by the zig-zag norm, so it suffices to consider the latter. An easy modification of the proof of \cite[Lemma 2.1]{colin-sandon} shows that any element $\phi \in \tcG$ sufficiently $C^1$-close to the identity can be represented as a product $\phi = f g$, where $f$ is positive, $g$ is negative, and both are $C^1$-close to the identity. By the proof of \cite[Lemma 2.1]{colin-sandon} again, both $f$ and $g$ are embedded, and so $\phi$ has zig-zag norm $\leq 2$.
\item The norms of Borman and Zapolsky coming from homogeneous quasi-morphisms
(see Section~\ref{sec:intro})
satisfy this assumption. Indeed, the quasi-morphisms $\phi$
they construct are \emph{monotone}, meaning that $\phi(g) \leq \phi(h)$ if $g \preceq h$. This implies that the corresponding conjugation-invariant norms $\mu$  are dominated by our norm $\nu$ in the sense that for each such $\mu$
there exists $K > 0$ such that $\mu(g) \leq K \cdot \nu(g)$ for all $g \in \tcG(V)$ (see \cite[Lemma 1.33]{BZ}).
\end{itemize}
\end{rem}

\begin{rem} \label{rem:compare} In the previous Remark
we mentioned comparison of the norms of
Colin-Sandon \cite{colin-sandon} with each other and the norm of
Zapolsky \cite{zap}
with ours. Let us now discuss comparison
of our norm $\nu$ with the zigzag norm, $\nu_{zig}$,
and oscillation norm, $\nu_{osc}$, of
Colin-Sandon. In fact the compatibility of $\nu_{osc}$ with the partial
order $\preceq$ \cite[Proposition 3.4]{colin-sandon} immediately implies
that
$\nu_{osc} \leq 2 \nu(f) + 1$. Comparison of our norm
with the zigzag norm is more subtle.
It is useful to consider the word norm $\mu_H$ defined
by a generating set $S$ consisting of isotopies with
Hamiltonian $H$ such that $-s < H < s$. It will not be conjugation-invariant
in general since $S$ is not, but we do have $\nu \leq \mu_H$ by Remark~\ref{rem-amplitude}.
The converse relationship between $\nu$ and $\mu_H$ on the other hand
is not yet clear; it
depends on how hard it is to simultaneously fulfill both conditions in
Remark~\ref{rem-amplitude}, that on $F_+$ and that on $F_-$, with a single Hamiltonian
$F = F_+ = F_-$. The norm $\mu_H$ can in
some cases be shown to dominate $\nu_{zig}$, i.e., $\nu_{zig} \leq K \mu_H$.
It can also be shown that the norm $\mu_H$ is {\it equivalent} to $\lceil \nu_{S} \rceil$,
i.e., each norm dominates the other,
where $\nu_{S}$ is the (pseudo-) norm of Shelukhin
\cite{shelukhin}, and $\lceil \cdot \rceil$ denotes the ceiling function.
Indeed, the inequality $\nu_{S} \leq \mu_H$ is immediate, hence $\lceil \nu_{S} \rceil \leq \mu_H$,
while an inequality in the converse direction follows by a re-parametrization trick (see proof of \cite[Lemma 5.1.C]{leonid-book}).
\end{rem}

The rest of Section~\ref{sec:obstr} is organized as follows. Theorem \ref{thm-discrete} is proved in Section \ref{subsect-lower}, and Theorem \ref{thm-frag-universal} in Section \ref{subsect-upper}. Beforehand, in Section \ref{subsect-alg} we recall some algebraic results used in those proofs. Finally, in Section \ref{subsection-sandtubes} we discuss a class of sub-domains of contact manifolds for which
analogous boundedness results can be obtained.

\subsection{Algebraic results}\label{subsect-alg}

The following definition is taken from \cite{BIP}.

\begin{defin}
Let $G$ be a group, and let $H \subset G$ be a subgroup. We say that an element $g \in G$ \emph{$m$-displaces} $H$ if the subgroups $$H, gHg^{-1}, g^2Hg^{-2}, \ldots, g^m H g^{-m} $$
pairwise commute.
\end{defin}

The geometric meaning of $m$-displacement in our context is as follows. Let $U \sub V$  be an open subset.
We say that a contactomorphism $\phi\in \cG(V)$ $m$-displaces $U$ if the subsets
\begin{equation*}
U, \phi(U), \ldots, \phi^m(U)
\end{equation*}
are pairwise disjoint. If this holds then $\phi$ $m$-displaces the subgroup $\cG(U)$ of $\cG(V)$. Similarly, if $\wtil{\phi} = \{\phi_t\} \in \tcG$ is a path such that $\phi_1$ $m$-displaces $U$ then $\wtil{\phi}$ $m$-displaces the subgroup $\tcG(U,V)$ of $\tcG(V)$.

Returning to the general algebraic setting, given a subgroup $H \subset G$ and an element $h$ in the commutator subgroup $[H, H]$, we denote by $cl_H(h)$ the commutator length of $h$, which is the minimal number of commutators needed to represent $h$ as a product of commutators. We will need the following result (see \cite{BIP}, Theorem 2.2). Suppose that $\mu$ is a conjugation-invariant norm on a group $G$. Let $H \subset G$ be a subgroup such that there exists $g \in G$ which $m$-displaces $H$. Then for any $h \in [H, H]$ with $cl_H(h) = m$ one has
\begin{equation}\label{eq-14}
\mu(h) \leq 14 \mu(g).
\end{equation}

Finally, we will use the following result of Tsuboi \cite{Tsuboi} dealing with contactomorphisms of finite smoothness.
Let $W$ be a connected contact manifold of dimension $2n+1$. For $1\leq r < \infty$, denote by $\cG^r(W)$ the identity component of the group of compactly supported  $C^r$-contactomorphisms of $W$, and by $\tcG^r(W)$ its universal cover. Moreover, for an open subset $X\subset W$, denote by $\tcG^r(X,W)$ the subgroup of $\tcG^r(W)$ containing those contact isotopies which are supported in $X$.

Tsuboi's theorem states that for $r\leq n+3/2$, the groups $\cG^r(W)$ and $\tcG^r(W)$ are perfect, i.e. equal to their commutator subgroups.  In particular if $X \subset W$ is a connected open subset, the groups $\cG^r(X)$ and $\tcG^r(X,W)$ are perfect. The latter group is perfect since it is an epimorphic image of the perfect group $\tcG^r(X)$ (similarly to \eqref{eq-chi} above).

\subsection{Discreteness}\label{subsect-lower}

In this section we prove Theorem \ref{thm-discrete} on discreteness. In what follows by an embedded open ball we mean the interior of an embedded closed ball. We use the following fact which we prove in Section \ref{subsect-upper} (see Example \ref{exam-space-port}).
\begin{lemma}\label{lem-cG-transitive}
Let $D_1, D_2 \subset V$ be Darboux balls and let $U$ be an open subset of $D_1$ such
that $\text{Closure}(U) \subset D_1$. Then there exists $\phi \in \cG(V)$ such that $\phi(U) \subset D_2$.
\end{lemma}

\begin{proof}[Proof of Theorem \ref{thm-discrete}]
First we note that (2) follows from (1). Indeed, denote by $\pi \colon \tcG \to \cG$ the natural projection. Assume that (1) holds, and let $\mu$ be a conjugation-invariant norm on $\tcG$. Define a conjugation-invariant norm $\mu_*$ on $\cG$ by $\mu_*(f) = \inf \{\mu(f') \,:\,{\pi(f')=f}\}$. Then by assumption $\mu_*$ is discrete. Observe that $\mu(f) \geq \mu_*(\pi f)$ for all $f\in \tcG$. Since for all $f \notin \pi_1(\cG)$, $\pi(f) \neq \id$, we get (2).

Next we prove (1). Assume the result does not hold. Fix an open ball $U$ with closure contained in a Darboux ball $D$, and a pair of elements
$\phi, \psi \in \cG(U)$ with $[\phi, \psi] \neq \id$. We claim
$\mu([\phi, \psi]) = 0$, a contradiction.

By assumption, for any $\eps > 0$ we can find $\theta \in \cG$ with $\mu(\theta) < \eps$. Since $\theta \neq \id$, $\theta$ moves some point, so there must exist an open ball $B \sub V$ such that $\theta(B) \cap B = \eset$.
Let $\eta \in \cG$ such that $\eta (U) \sub B$ (which exists by Lemma \ref{lem-cG-transitive}).
Then $\eta^{-1} \theta \eta$ displaces $\cG(U)$ and hence by \eqref{eq-14}
$$
\mu([\phi, \psi]) \leq 14 \mu(\eta^{-1} \theta \eta) = 14 \mu(\theta) < 14 \eps.
$$
Our claim follows.
\end{proof}

\subsection{Boundedness}\label{subsect-upper}

In this section we prove Theorem \ref{thm-frag-universal}.

\begin{defin}\label{def-portable}
	An {\it open} connected contact manifold $(V, \xi)$ is called \emph{contact portable} if there exists a
	connected compact set $V_0 \sub V$  and a contact isotopy $\{P_t\}$ of $V$, $t \geq 0$, $P_0=\id$ such that the following hold:
	\begin{itemize}
		\item The set $V_0$ is an \emph{attractor} of $\{P_t\}$, i.e. for every compact set $K \sub V$
		and every neighborhood $U_0 \supset V_0$ there exists some $t > 0$
		such that $P_t(K) \sub U_0$.
		\item There exists a contactomorphism $\theta $ of $V$ displacing $V_0$.
	\end{itemize}
\end{defin}

\medskip
\noindent
Note that $\theta$ is not assumed to be compactly supported.
Definition \ref{def-portable} is a contact version of the notion of portable manifold defined in \cite{BIP}.

\medskip
\noindent
\begin{exm}\label{exam-space-port}{\rm An example of a contact portable manifold is $\R^{2n+1}$, equipped with the standard contact structure given by the
		kernel of the $1$-form $\alpha = dz - y dx$. Here we use the coordinates $(x, y, z) \in \R^n \times \R^n \times \R$.
		The contact isotopy is given by
		$$
		P_t \colon (x, y, z) \mapsto (e^{-t} x, e^{-t}y, e^{-2t} z).
		$$
		The attractor $V_0$ can be taken to be the closed ball $\{ |x|^2 + |y|^2 + z^2 \leq 1\}$ and the contactomorphism $\theta$ can be given, for example, by
		$$
		\theta(x, y, z) = (x, y , z + 3).
		$$
		Similarly, any Euclidean ball $\{|x|^2+|y|^2+z^2 <R\}$ is contact portable.
		
		In particular, this example shows that every compact subset of $\R^{2n+1}$ can be contact
		isotoped into an arbitrary small neighborhood of the origin. Applying an appropriate cut-off function to the generating contact Hamiltonian, this can be done inside a Darboux chart in an arbitrary contact manifold.
		This, together with the transitivity of $\cG$,
		proves Lemma \ref{lem-cG-transitive}.}
\end{exm}

The proof of the following proposition is analogous to that of Theorem 1.17 in \cite{BIP}.

\begin{prop}\label{prop-port-bd}
	Let $(V, \xi)$ be a contact portable manifold. Then any conjugation-invariant norm on $\cG(V)$ or $\tcG(V)$,
	which is bounded on a $C^1$-neighborhood of the identity, is bounded.
\end{prop}

We prove the case of $\tcG(V)$; the proof for $\cG(V)$ is similar.

\begin{proof}
	Let $\mu$ be such a norm, and fix a constant $\delta >0 $ and a
	$C^1$-neighborhood $\mathcal{V}\subset \tcG$ of the identity such that $\mu(g)<\delta$ for $g\in \mathcal{V}$.
	Let $V_0$, $P_t$ and $\theta$ be as in Definition \ref{def-portable}.
	We can find a small connected neighborhood $U$ of $V_0$ with compact
	closure such that $\theta(U) \cap U = \emptyset$.
	First we show that $\mu$ is bounded on the subgroup $H := \tcG(U,V)$ of $\tcG$.
	
Indeed, let $U' := \theta(U)$. Then  $U'$ is a neighborhood of $\theta(V_0)$, which
	is an attractor for the isotopy $\{ g_t = \theta \circ P_t \circ \theta^{-1}\}$. Therefore,
	for some $T > 0$, $g_T (U \cup U') \sub U'$.
	Truncating the contact Hamiltonian generating $\{g_t\}$
	and re-parametrizing gives a contact isotopy $\psi = \{\psi_t\} \in \tcG(V)$ such
	that $\psi_1 (U \cup U') \sub U'$. We claim $\psi_1$ $m$-displaces $U$
	for all $m \geq 0$. Indeed $U$ lies in the complement of $U'$, so
	$\psi_1 (U) \subset U' \setminus \psi_1 (U')$
	and $\forall k \in \N$,
	$\psi_1^k (U) \subset \psi_1^{k-1}(U') \setminus \psi_1^k(U')$, which implies the $\psi_1^k (U)$
	are pairwise disjoint.
	
	Now, by Tsuboi's theorem any $h \in H$ can be written as a product
	$h=h_1 \cdot \ldots \cdot  h_m$, $h_i=[\sigma_i,\tau_i]$ where $\sigma_i, \tau_i \in \tcG^1(U,V)$.
	Since the $2m$-fold product of $\tcG(U,V)$ is dense in the $2m$-fold product of $\tcG^1(U,V)$
	and the product of $m$ commutators defines a continuous map to $\tcG^1(U,V)$,
	there is some $g \in \tcG(U,V)$ satisfying $g^{-1}h \in \mathcal{V}$ such that
	$g=g_1 \cdot \ldots \cdot g_m$ where each $g_i$ is a product of commutators of
	elements of $\tcG(U,V)$.
	In particular, $cl_H(g)\leq m$, and hence
	$$
	\mu(g) \leq 14 \mu(\psi).
	$$
	But then
	$$
	\mu(h) \leq \mu(g) + \mu(g^{-1}h) \leq 14 \mu(\psi) + \delta=:C.
	$$
	This proves that $\mu$ is bounded on $H$.
	
	Now, given $f = \{f_t\} \in \tcG(V)$, let $K$ be a compact set such
	that $\cup_t \supp f_t \sub K$. There exists $T$ such that $P_T (K) \sub U$.
	As before,
	truncating and re-parametrizing the contact
	Hamiltonian which generates $\{P_t\}_{t \in [0, T]}$, we can produce
	$\eta \in \tcG(V)$ with $\eta (K) \sub U$. Then $\eta \circ f \circ  \eta^{-1} \in H =\tcG(U,V)$ and
	so,
	$$
	\mu(f) = \mu (\eta \circ f \circ  \eta^{-1}) \leq C.
	$$
\end{proof}

\begin{rem}\label{rem-bd-pseudo}
Observe that the proof of Proposition \ref{prop-port-bd} works equally well for a conjugation-invariant pseudo-norm; the non-degeneracy of $\mu$ was never used.
\end{rem}

We can now prove the maximality of the fragmentation norm. As above, we prove it for $\tcG(V)$; the proof for $\cG(V)$ is similar.
\begin{proof}[Proof of Theorem \ref{thm-frag-universal}]
Let $\mu$ be a conjugation-invariant norm on $\tcG(V)$ which is bounded on a $C^1$-neighborhood of the identity. Fix a Darboux ball $B \subset V$. By Proposition \ref{prop-port-bd}, $\mu$ is bounded on $\tcG(B, V)$ - indeed, pulling back $\mu$ by the epimorphism $\tcG(B) \to \tcG(B,V)$ (recall \eqref{eq-chi}) yields a conjugation-invariant pseudo-norm on $\tcG(B)$, which is bounded by Remark \ref{rem-bd-pseudo} and Example \ref{exam-space-port}, and so $\mu$ is bounded on $\tcG(B,V)$, say by $C>0$. Now, let $f \in \tcG(V)$. Write $f=h_1 \cdots h_N$, where $N = \mu_F(f)$, and each $h_i$ is represented by an isotopy supported in a Darboux ball $B_i \subset V$. By Lemma \ref{lem-cG-transitive}, one can find contact isotopies mapping each $\supp h_i$ into $B$, and so each $h_i$ is conjugate to a contact isotopy supported in $B$. Then for any $1\leq i \leq N$, $\mu(h_i) \leq C$. We get $\mu(f) \leq CN = C \nu_F(f)$.
\end{proof}

\subsection{Boundedness on sub-domains}\label{subsection-sandtubes}

\begin{defin}
Suppose $(V, \xi)$ is a contact manifold and $V' \subset V$ an open subset.
Then we say that $V'$ is a {\it portable sub-domain} of $V$ if
there exists a compact set $V_0 \sub V'$,
and a contact isotopy $\{P_t\}_{t \in \R}$ of $V$ such that the following hold:
\begin{itemize}
\item For every compact set $K \sub V'$ and every neighborhood $U_0 \supset V_0$
there exists some $t > 0$
such that $P_t(K) \sub U_0$.
\item There exists a contactomorphism $\theta $ supported in $V'$ displacing $V_0$.
\end{itemize}
\end{defin}
Observe that a portable sub-domain $V' \subset V$ need not be a
contact portable
manifold, since we allow more ``squeezing room'' (the isotopy $P_t$ may have support outside $V'$).
By the same argument as in the proof of  Proposition~\ref{prop-port-bd}, we have the following result:

\begin{prop}\label{prop-portdomain-bd}
Let $(V, \xi)$ be a contact manifold and $V' \subset V$ be a portable sub-domain of $V$.
Then any conjugation-invariant norm on $\cG(V)$ (resp. $\tcG(V)$) which is bounded
on a $C^1$-neighbourhood of the identity is necessarily
bounded on $\cG(V')$ (resp. on $\tcG(V',V))$.
\end{prop}

\medskip
\noindent
\begin{exm}\label{exam-portdomain}
{\rm
Consider the contact manifold $V= \R^{2n} \times S^1$ equipped with the contact structure
$\xi=\text{Ker}(dt-\alpha)$, where $\alpha=\frac{1}{2}(pdq-qdp)$. In what follows we assume that
$n \geq 2$. Put $\cU(r) := B^{2n}(r) \times S^1$,
where $B^{2n}(r)$ stands for the ball $\{\pi(|p|^2 +|q|^2) < r\}$. In \cite{sandon-bi-invt} Sandon defined a conjugation
invariant norm on $\cG(V)$  which is bounded on all subgroups $\cG(\cU(r))$. We claim that every conjugation-invariant norm on $\cG(V)$ is necessarily
bounded on $\cG(\cU(r))$ if $r < 1$.

To prove the claim, take any Hamiltonian symplectomorphism $\theta$ supported
in $B^{2n}(r) \subset \R^{2n}$ which displaces the origin. Let $r' < r$ be sufficiently small
so that $B^{2n}(r')$ is also displaced and let
$\widehat\theta \in \cG(\cU(r))$ be the lift of $\theta$ to a contactomorphism of $V$
supported in $\cU(r)$.
We have $\widehat\theta(\cU(r')) \cap \cU(r') =\emptyset$. Further,
by the Squeezing Theorem \cite[Theorem 1.3]{EKP06}
there exists $P \in \cG(V)$ such that $P(\cU(r)) \subset \cU(r')$. Therefore
$\cU(r)$ is a portable sub-domain of $V$, and the claim follows from
Proposition~\ref{prop-portdomain-bd}.}\end{exm}

\medskip

In fact, the above argument can be applied more generally.
Recall that a symplectic manifold $(M^{2n},\omega=d\alpha)$ is called {\it Liouville}
if it admits a vector field $v$ and a compact $2n$-dimensional submanifold $\overline{\cU}$ with
connected boundary $Q=\partial \overline{\cU}$ with the following properties:
\begin{itemize}
\item $i_\eta \omega = \alpha$. This yields that the flow $\eta_t$ of $v$ is conformally symplectic;
\item $v$ is transversal to $Q$.
\end{itemize}
One can show that $(Q,\text{Ker}(\alpha))$ is a contact manifold and for all specific choices
of $\overline {\mathcal U}$ these are naturally contactomorphic.
We refer to $(Q, \text{Ker}(\alpha))$ as {\it the ideal contact boundary} of $M$. The set
$C:= \bigcap_{t>0} \eta_{-t}(\overline{\cU})$ is called {\it the core} of $M$.

Consider now the contact manifold $V=M \times S^1$ equipped with the contact form $\lambda =dt-\alpha$. 
Put $\cU(r) := \eta_{\,\log r}(\cU) \times S^1$, where $\cU$ is the interior of $\overline{\cU}$.

\begin{prop}\label{prop-nonord}
Suppose that the ideal contact boundary of $(M,d\alpha)$ is non-orderable and
the core $C$ is displaceable by a Hamiltonian diffeomorphism in its arbitrary small neighborhood. Then
there exists $r_0 >0$ so that any conjugation-invariant norm on $\cG(V)$ (resp. $\tcG(V)$) which is bounded
on a $C^1$-neighbourhood of the identity is necessarily
bounded on $\cG(\cU(r))$ (resp. $\tcG(\cU(r),V))$ for all positive $r < r_0$.
\end{prop}

\begin{proof} By  \cite[Theorem 1.19]{EKP06} there is some $r_0 > 0$ such that (by iterating the Theorem enough times),
one can obtain an isotopy which squeezes
$\cU(r_0)$ arbitrarily close to $C \times S^1$, in fact within\footnote{See the argument in Remark 1.23 of
that paper and the proof of Theorem 1.3 on the same page.} some larger $\cU(r') \subset V$.
Since the core $C$ is Hamiltonian displaceable in its arbitrarily small neighborhood,
the set $C \times S^1$ is displaceable in its arbitrarily small neighborhood by a contact isotopy of $V$.
It follows that for $0< r < r_0$ the set  $\cU(r)$ is a portable sub-domain of $V$ and
hence  by Proposition~\ref{prop-portdomain-bd} any conjugation-invariant norm on $\cG(V)$ (resp. $\tcG(V)$) is bounded on  $\cG(\cU(r))$ (resp. on $\tcG(\cU(r),V))$.
\end{proof}

\medskip
\noindent
An important class of Liouville manifolds is formed by complete Stein manifolds, that
is by K\"{a}hler manifolds $(M,J,d\alpha)$ admitting a proper bounded from below Morse function $F$ with
$\alpha = JdF$ and $\eta=-\text{grad} F$, where the gradient is taken with respect to the metric
$d\alpha(\cdot,J\cdot)$. (By a result of Eliashberg \cite{E90} these manifolds admit an alternative description
as Weinstein manifolds provided $\dim M \geq 3$). For a generic $F$, the core of $M$ is
an isotropic $CW$-complex of dimension $n-k$ with $k \in [0, n]$ (see \cite{EG, BiCi}). We say that $M$ is $k$-{\it subcritical}
if $k \geq 1$ and {\it critical} if $k=0$.

\medskip

\noindent
The assumptions of Proposition \ref{prop-nonord} implying boundedness of
suitable conjugation-invariant norms on some $\cG(\cU(r))$ hold true, for instance, when the Liouville manifold $M$ is $k$-critical with $k \geq 2$ : indeed, the ideal contact boundary is non-orderable by Theorem 1.16 of \cite{EKP06}, while the core $C$ in this case is Hamiltonian displaceable in its arbitrary small neighborhood (cf. \cite[Section 3]{BiCi}).

\medskip

\noindent
On the other hand, for some critical Liouville manifolds $M$ it is known that the core is stably non-displaceable. For instance, this is true for cotangent bundles of closed manifolds, where the core
can be taken as the zero section. By Theorems \ref{thm-2}, \ref{thm-3} and Corollary \ref{cor-stb-unbd-tG} this implies the conjugation invariant norm $\mu$ on $\tcG(V)$ is well defined and stably unbounded when restricted to $\tcG(\cU(r))$ for all $r>0$. Here the 1-periodic Reeb flow on $V$ is associated to the form
$\lambda$ and is given by rotation along the $S^1$-factor so leaves $\cU(r)$ invariant.

\medskip

\noindent
We thus see a dichotomy between boundedness of
conjugation-invariant norms in small $\cU(r) \subset M \times S^1$ in the case of sub-critical $M$, and stable unboundedness of the norm $\nu$ on any $\cU(r)$
in the case of certain critical $M,$ those with stably non-displaceable core. It is natural to ask:

\begin{question} {\rm Is the norm $\nu$ well-defined
and stably unbounded on any $\cU(r) \subset M \times S^1$
for all critical $M$?}
\end{question}

\medskip
\noindent
 It could
be that $V = M \times S^1$ is orderable for every Liouville manifold
$(M, d\alpha)$ which would confirm at least well-definedness of the norm $\nu$
for all $V$ of this kind.
This is known, for instance, for some critical $M$ such as cotangent bundles, as well as for some subcritical $M$ such as linear spaces due to the thesis of Sandon \cite{sandon-thesis}. P. Albers
pointed out that the methods of \cite{albers-merry} should prove the result for general $M$.

\medskip
\noindent
Regarding boundedness vs. stable unboundedness, however,
little is
known so far. On the one hand, in case one
wished to prove boundedness by using squeezing as in Proposition \ref{prop-nonord} above, it is unknown  whether the ideal contact boundary of critical $M$ is orderable or not. On the other hand,
it is unlikely that the technique of Section \ref{sec:constr} based on stable intersection property might be applicable to proving stable unboundedness for subcritical manifolds with $\dim M \geq 2$.
Indeed, by  \cite[Theorem 6.1.1]{BiCi1}, there are no ``hard" symplectic obstructions to Hamiltonian displacement of a compact subset of $SV=M \times \R_+ \times S^1$ from a given closed
subset, so existence of sets with stable intersection property is quite problematic.
It would be interesting to explore this point further.

\section{Acknowledgements} The results of the present paper have been presented at the
AIM workshop `Contact topology in higher dimensions' in May, 2012. We thank the organizers,
J.~Etnyre, E.~Giroux and K.~Niederkrueger, for the invitation and AIM for the excellent research
atmosphere. We are grateful to  V.~Colin, Y.~Karshon, and S.~Sandon for very helpful
discussions during the workshop. We also thank V.~Colin, S.~Sandon, F.~Zapolsky, P.~Albers and W.~Merry for communicating early versions of \cite{colin-sandon},\cite{zap} and \cite{albers-merry} respectively with us. We thank M.~Usher for useful discussions. L.P. and D.R. were partially supported by the Israel Science Foundation grants 509/07 and 178/13,
by the National Science Foundation grant DMS-1006610
and by the European Research Council advanced grant 338809.


\begin{thebibliography}{99}
\bibitem[AF08]{AF08} P. Albers, U. Frauenfelder, {\it A nondisplaceable Lagrangian torus in $T^\ast S^2$}. Comm. Pure and Appl. Math. {\bf 61} (2008), 1046--1051.

\bibitem[AM13]{albers-merry} P. Albers, W. Merry, {\it Translated points and Rabinowitz Floer homology}. J. Fixed Point Theory Appl. {\bf 13}(1) (2013), 201--214. 

\bibitem[B97]{banyaga} A. Banyaga, {\it The structure of classical diffeomorphism groups}. Mathematics and its Applications, 400. Kluwer Academic Publishers Group, Dordrecht, 1997.

\bibitem[BC02]{BiCi} P. Biran, K. Cieliebak, {\it Lagrangian embeddings into subcritical Stein manifolds}. Israel J. Math.  {\bf 127}  (2002), 221--244.

\bibitem[BC02a]{BiCi1} P. Biran, K. Cieliebak, {\it Symplectic Topology on Subcritical Manifolds}.
Comm. Math. Helv. {\bf 76} (2002), 712--753.

\bibitem[BZ15]{BZ} S. Borman, F. Zapolsky, {\it Quasi-morphisms on contactomorphism groups and contact rigidity}. Geom. Topol., {\bf 19}(1) (2015), 365--411.

\bibitem[BIP08]{BIP} S. Burago, D. Ivanov, L. Polterovich, {\it Conjugation-invariant norms on groups of geometric origin}. Groups of diffeomorphisms, Adv. Stud. Pure Math., {\bf 52} (2008), 221--250.

\bibitem[CS12]{colin-sandon} V. Colin and S. Sandon, {\it The discriminant length for contact and Legendrian isotopies}. J. Eur. Math. Soc. (JEMS) {\bf 17}(7) (2015), 1657--1685. 

\bibitem[E92]{E92} Y. Eliashberg, {\it Contact 3-manifolds twenty years since J. Martinet's work}.
Ann. Inst. Fourier (Grenoble) {\bf 42} (1992), 165--192.

\bibitem[E90]{E90} Y. Eliashberg, {\it Topological characterization of Stein manifolds of dimension $>2$}.
Internat. J. Math. {\bf 1}(1) (1990), 29--46.

\bibitem[EG91]{EG} Y. Eliashberg and M. Gromov,
{\it Convex symplectic manifolds.} In {\it Several complex variables and complex geometry, Part 2 (Santa Cruz, CA, 1989),}
Proc. Sympos. Pure Math. {\bf 52}, Part 2, Amer. Math. Soc., Providence, RI, (1991), 135--162.

\bibitem[EHS95]{EHS} Y. Eliashberg, H. Hofer, D. Salamon, {\it Lagrangian intersections in
contact geometry}. Geom. Funct. Anal. {\bf 5} (1995), 244--269.

\bibitem[EP00]{EP}  Y. Eliashberg and L. Polterovich, {\it Partially ordered groups and geometry
of contact transformations}. Geom. Funct. Anal. {\bf 10} (2000), 1448--1476.

\bibitem[EKP06]{EKP06} Y. Eliashberg, S.S. Kim, L. Polterovich, {\it Geometry of contact transformations and domains: orderability versus squeezing}. Geom. Topol. {\bf 10} (2006), 1635--1747.

\bibitem[G11]{Ginzburg} V. Ginzburg, {\it On Maslov class rigidity for coisotropic submanifolds.}
Pacific J. of Math. {\bf 250} (2011), 139--161.

\bibitem[Giv90]{givental90} A. Givental, {\it The nonlinear Maslov index}. Geometry of low-dimensional manifolds, 2 (Durham, 1989),
London Math. Soc. Lecture Note Ser. {\bf 151} (1990), 35��-43.

\bibitem[G85]{G85} M. Gromov, {\it Pseudoholomorphic curves in symplectic manifolds}. Invent. Math. {\bf 82} (1985), 307--347.

\bibitem[LS94]{LS2} F. Laudenbach and J.-C. Sikorav, {\it Hamiltonian disjunction and limits of Lagrangian
submanifolds}. Internat. Math. Res. Notices {\bf 4} (1994), 161--168.


\bibitem[PPS03]{PPS-MMJ} G.P. Paternain, L. Polterovich and K.F. Siburg,
{\it Boundary rigidity for Lagrangian submanifolds, non-removable
intersections, and Aubry-Mather theory.} Mosc. Math. J. {\bf 3}(2)
(2003), 593--619.

\bibitem[P01]{leonid-book} L. Polterovich, {\it The geometry of the group of symplectic diffeomorphism} Lectures in Mathematics, ETH-Zurich, Birkhauser, Basel, 2001.

\bibitem[R10] {Ryb} T. Rybicki, {\it Commutators of contactomorphisms}. Adv. Math. {\bf 225}(6) (2010), 329--336.

\bibitem[R12]{Ryb12} T. Rybicki, {\it Bi-invariant metric on the strict contactomorphism group}. {\tt arXiv:1202.5897}, (2012). Now withdrawn.

\bibitem[R13]{Ryb13} T. Rybicki, {\it Hofer metric from the contact point of view}. {\tt arXiv:1304.1971}, (2013). Now withdrawn.

\bibitem[S10]{sandon-bi-invt} S. Sandon, {\it An integer valued bi-invariant metric on the group of contactomorphisms of $R^{2n}\times S^1$}. J. Topol. Anal. {\bf 2} (2010) 327--339.
\bibitem[S11]{sandon-thesis}  S. Sandon, {\it Contact homology, capacity and non-squeezing in $\mathbb{R}^{2n} \times S^1$ via generating functions}. Ann. Inst. Fourier (Grenoble) {\bf 61} (2011), 145--185.

\bibitem[S12]{sandon-leaf} S. Sandon,
{\it On iterated translated points for contactomorphisms of $R^{2n+1}$ and $R^{2n} \times S^1$}. Internat. J. Math. {\bf 23}(2) (2012). 

\bibitem[Sh14]{shelukhin}E. Shelukhin,
{\it The Hofer norm of a contactomorphism}. {\tt arXiv:1411.1457}, (2014).

\bibitem[S76]{Sul}  D. Sullivan, {\it Cycles for the dynamical study of foliated manifolds and complex manifolds}.
Invent. Math. {\bf 36} (1976), 225--255.

\bibitem[T08]{Tsuboi} T. Tsuboi, {\it On the simplicity of the group of contactomorphisms}, Groups of diffeomorphisms,
Adv. Stud. Pure Math. {\bf 52} (2008), Math. Soc. Japan, Tokyo, 491--504.

\bibitem[U14]{usher} M. Usher, {\it Hofer geometry and cotangent fibers}. J. Symplectic Geom.
{\bf 12}(3) (2014), 619--656.

\bibitem[Z13]{zap} F. Zapolsky, {\it Geometric structures on contactomorphism groups and
contact rigidity in jet spaces}. Internat. Math. Res. Notices, {\bf 2013}(20) (2013), 4687--4711.





\end{thebibliography}
\end{document}